\documentclass[leqno]{amsart}

\usepackage{amsmath}
\usepackage{amscd}
\usepackage{amsthm}
\usepackage{amssymb}
\usepackage{latexsym}

\usepackage{enumerate}

\newtheorem{thm}{Theorem}[section]
\newtheorem{prop}{Proposition}[section]
\newtheorem{lem}{Lemma}[section]

\newtheorem{rem}{Remark}[section]

\newtheorem*{thmm}{Theorem~ALW}

\allowdisplaybreaks

\title{Positive and negative growth rates of measures and pointwise bounds of solutions of eigenvalue equations${}^{*}$}

\author{Hironori Kumura${}^{**}$}

\thanks{${}^{*}$\
{\it Key Words and Phrases}; Laplace-Beltrami operator, eigenfunction, exponential decay, spectrum}

\thanks{${}^{**}$\ 
smhkumu@ipc.shizuoka.ac.jp\\
Department of Mathematics, Shizuoka University, 
Shizuoka 422-8529, Japan} 
\date{revised form; April, 2012}

\begin{document}

\maketitle

\begin{abstract}
This paper, focusing on the growth rate of the measure, gives pointwise bounds of solutions of eigenvalue equations of the Laplace-Beltrami operator on noncompact Riemannian manifolds. 
\end{abstract}
\section{Introduction}

The Laplace-Beltrami operator $\Delta_g$, defined on $C^{\infty}_0(M)$, of a noncompact complete Riemannian manifold $(M, g)$ is essentially self-adjoint on $L^2(M, v_g)$; the analytic structures of its self-adjoint extension, denoted also by $\Delta_g$, and the geometries of $(M, g)$ are closely related to each other, and their relationship has been studied by several authors from various points of view. 
The main purpose of this paper is to study pointwise bounds of solutions of eigenvalue equations in cases that the growth rate of the measure is both positive and negative. 

First, we shall recall an earlier work due to Li-Wang on the decay of harmonic functions on noncompact Riemannian manifolds. 
Let $(M, g)$ be an $n$-dimensional connected noncompact complete Riemannian manifold. 
In the sequel, the following notations will be used: 
$\sigma(-\Delta_g)$, $\sigma_{\rm ess}(-\Delta_g)$, and $\sigma_{\rm pp}(-\Delta_g)$ respectively stand for the spectrum of $-\Delta_g$, the essential spectrum of $- \Delta_g$, and the set of eigenvalues of $-\Delta_g$. 
Moreover, let $v_g$ denote the Riemannian measure of $(M, g)$. 
Li-Wang improved the Agmon's method \cite{A-book} used for the Schr\"odinger operators on $\mathbb{R}^n$, and, among other things, proved the following: 
\begin{thmm}[Agmon \cite{A-book}, Li-Wang \cite{L-W}]
Let $(M, g)$ be a complete noncompact Riemannian manifold, and $E$ be an unbounded connected component of $M \backslash B(p_0, R)$ for some $R>0$. 
Let $\lambda_0(E)$ denote the minimum of the spectrum of the Dirichlet Laplacian $-\Delta_g^D|_E$ on $E$. 
Assume that $0 \le \lambda < \lambda_0(E)$, and that $u$ is a solution of the equation $-\Delta_g u - \lambda u = 0$ on $E$ with $\int_E u^2 \, dv_g <\infty$, where $\lambda$ is a constant. 
Then, there exists a constant $C_1 > 0$ such that 
\begin{align}
 \int_{E(R, R+1)} u^2 \, dv_g \le C_1 \exp \Big( - 2 \sqrt{\lambda_0(E) - \lambda} \, R \Big) \quad {\rm for~} R \gg 1, 
\end{align}
where we set $E(R, R+1) := \{ x \in E \mid R < {\rm dist}_g(\partial E, x) < R+1 \}$. 
\end{thmm}

Note that the inequality $(1)$ was proved only for $L^2$-harmonic functions in \cite{L-W}. 
However, it is not hard to see that their arguments hold good for $L^2$-solutions of eigenvalue equations by slight modifications; see [11, Lemma 1.1] and its proof. 

In general, $\lambda_0(E(R, \infty)) < \min \sigma_{\rm ess} (-\Delta_g|_{E})$ for $R>0$ (see Proposition 6.1 below, which shows that, there exists a rotationally symmetric manifold satisfying this inequality for any $R>0$: see also Remark 6.1). 
Hence, Theorem ALW is not so sharp. 
Moreover, it seems to be difficult to convert this integral bound (1) into the pointwise bounds of eigenfunctions. 
In its character, Theorem ALW cuts, as it were, lengthwise the end $E$ of $(M, g)$. 
On the other hand, this study will {\it track the growth rate of the measure $v_g$ along geodesic rays}, and investigate pointwise bounds for solutions $u$ of $-\Delta_g u + \lambda u = 0$ on $E$ by constructing barriers near the infinity. 

\begin{prop}
\begin{sloppypar}
Let $(M,g)$ be an $n$-dimensional noncompact Riemannian manifold and $U$ be an {\rm (}not necessarily relatively compact\,{\rm )} open subset of $M$. 
Assume that the boundary $\partial U$ is compact, connected, and $C^{\infty}$ and that the outward normal exponential map $\exp_{\partial U}^{\perp} : N^{+}(\partial U) \to E : =M \backslash U$ induces a diffeomorphism, where $N^{+}(\partial U) := \{ v\in TM|_{\partial U} \mid v {\rm ~is~outward~normal~to~}\partial U \}$. 
Let $\nu$ be the outward unit normal vector field along $\partial E$, and denote $r(x) := {\rm dist}_g(\partial U, x)$ for $x \in E$. 
Then, $[0, \infty) \times \partial E \ni (r, w) \mapsto \exp _w(r\nu_w) \in E$ is a coordinates on $E$. 

Let $c>0$ be a constant, and assume that $u \in L^2 (E, v_g)$ is a solution of the eigenvalue equation 
\begin{align}
( -\Delta_g - \lambda )u = 0 \quad {\rm on}~~E,
\end{align}
where $0 \le \lambda < \frac{c^2}{4}$ is a constant. 
For simplicity, we shall set $\alpha := \frac{c}{2}+\sqrt{\frac{c^2}{4}-\lambda}$. 
\begin{enumerate}[{\rm (i)}]
\item If $\displaystyle \liminf_{r \to \infty} \Delta_g r \ge c$ and if $\beta_1$ is a constant satisfying 
$$
  \frac{c}{2}-\sqrt{\frac{c^2}{4}-\lambda} < \beta_1 < \frac{c}{2}+\sqrt{\frac{c^2}{4}-\lambda}, 
$$
then $\exp (- \beta_1 r)$ is a supersolution of $(2)$ near the infinity of $E$. 
In particular, there exists a constant $C_1>0$, depending on $\beta_1$, such that 
$$
  |u(x)| \le C_1 \exp \big(-\beta_1 r(x) \big) \quad {\rm for}~~x \in E.
$$
\item If $\displaystyle \lim_{r \to \infty} \Delta_g r = c$ uniformly with respect to $w \in \partial E$ and if $\beta_2$ is a constant satisfying either 
$$
  0 \le \beta_2 < \frac{c}{2} - \sqrt{\frac{c^2}{4}-\lambda} \quad {\rm or} \quad \beta_2 > \frac{c}{2} + \sqrt{\frac{c^2}{4} - \lambda}, 
$$
then $\exp (- \beta_2 r)$ is a subsolution of $(2)$ near the infinity of $E$. 
\item If there exist constants $\beta_3 > 0$ and $\varepsilon > 0$ satisfying $c - \displaystyle \frac{\beta_3}{r(\log r)^{1 + \varepsilon}} \le \Delta_g r$ for $r \gg 1$, then there exists a constant $C_2 > 0$ such that 
$$
  |u(x)| \le C_2 \exp \big( -\alpha r(x) \big) \quad {\rm for}~~x \in E.
$$
\end{enumerate}
\end{sloppypar}
\end{prop}
Note that, for example, (i), and (iii) in Proposition $1.1$ are sharp. 
Indeed, for example, there exists a rotationally symmetric manifold $(\mathbb{R}^n, g:= dr^2 + f(r)^2g_{S^{n-1}(1)})$ satisfying $c - \frac{\beta}{r \log (r+1) +1} \le \Delta_g r$ for $r \gg 1$ and a solution of $( - \Delta_g - \lambda )u = 0$ satisfies $u = (\log r)^{\theta} \exp \{ - \alpha r \}$ for $r\gg 1$, where $\beta > 0$ and $\theta = \theta (\beta, n, \lambda)>0$ are constants. 
Hence, we cannot adopt $\varepsilon = 0$ in Proposition 1.1 (iii). 
Various other bounds under various other growth rate conditions of the measure can be obtained, as is seen in Proposition 2.1 below. 

Note also that, our method of proof of Proposition 1.1 requires the assumption that the outward normal exponential map is a diffeomorphism, but if the $(n-1)$-dimensional measure of ${\mathcal Cut} (\partial E) \backslash {\mathcal Focal}(\partial E)$ is zero, this assumption of diffeomorphism is not required and the assertion of Proposition 1.1 holds good (see Lemma 4.1 below). 
Here, ${\mathcal Focal}(\partial E)$ stands for the focal points of $\partial E$ in $E$.

Next, we shall mention Theorem 1.1 below, which, in brief, states that the same assertions as those in Proposition 1.1 and Proposition 2.1 below hold on angular regions of Cartan-Hadamard manifolds.  


\begin{thm}
Let $(M, g)$ be an $n$-dimensional complete Riemannian manifold and $U_M$ be an open subset of $M$. 
Assume that $\min \sigma_{\rm ess}(-\Delta_g) >0$. 
Assume that, there exist an $n$-dimensional Cartan-Hadamard manifold $(N, h)$ satisfying its sectional curvature $K_h \le -\kappa_0$ for some constant $\kappa_0 >0$, a unit tangent vector $v_1^N \in T_qN$, and constants, $\theta_0 >0$, $0 \le \lambda < \min \sigma_{\rm ess}(-\Delta_g)$ and $c>0$, such that the following three conditions hold\,{\rm :} 
\begin{enumerate}[{\rm $1.$}]
\item there exists an isometry $\varphi : \exp_q^N \big( {\mathcal R}(v_1^N, \theta_0) \big) \xrightarrow{\sim} U_M$. Here, $\exp_q^N$ stands for the exponential map of $(N, h)$ at $q$, and ${\mathcal R}(v_1^N, \theta)$ denotes an open cone in $T_qN$ defined by ${\mathcal R}(v_1^N, \theta_0) := \{ v \in T_qN \mid \measuredangle (v_1^N, v) < \theta_0 \}${\rm ;} 
\item there exists a large constant $b \gg \kappa_0$ satisfying ${\rm Ric}_g \ge -(n-1)b$ on $M${\rm ;} 
\item $\frac{c}{2}-\sqrt{\frac{c^2}{4}-\lambda} < \sqrt{\min \sigma_{\rm ess}(-\Delta_g) - \lambda}$. 
\end{enumerate}
Assume that $u$ is a $L^2(M, v_g)$-solution of $- \Delta_g u - \lambda u = 0$ on $M$. 
On these understanding, if we replace $E$ with $\varphi \left(\exp_q^N \big( {\mathcal R}(v_1^N, \theta_1) \big)\right)$, then the same assertions as those stated in Proposition $1.1$ and Proposition $2.1$ below hold on $\varphi \left(\exp_q^N \big( {\mathcal R}(v_1^N, \theta_1) \big)\right)$ for any fixed $\theta_1 \in (0, \theta_0)$. 
\end{thm}
For a ground state, we can also show various lower bounds which are mentioned as upper bounds for $u$ in Proposition 2.1 below. 
But, we shall only give the following result here, because other versions of lower bounds are easily formulated and proved in similar manner: 
\begin{thm}
Let $(M, g)$ be an $n$-dimensional noncompact complete Riemannian manifold and $U$ be a {\rm relatively compact} open subset of $M$ with $C^{\infty}$ boundary $\partial U$. 
We write $r:={\rm dist}_g(\partial U,*)$ on $M\backslash U$. 
Let $c > 0$ be a positive constant, and assume that $\lambda_0 := \min \sigma(-\Delta_g) < \frac{c^2}{4}$ and $\lambda_0 \in \sigma_{\rm pp}(-\Delta_g)$. 
Let $u$ be a ground state, that is, a positive solution of $(-\Delta_g - \lambda_0)u =0$ on $M$, and set $\alpha_0 := \frac{c}{2} + \sqrt{\frac{c^2}{4} - \lambda_0}$. 
Assume that there exist constants $r_0 \ge 0$, $\delta > 0$, and $C_1 >0$ satisfying 
\begin{align}
  \Delta_g r \le c + \frac{C_1}{r \big( \log (r+1) \big)^{1+\delta}+1} 
  \quad {\rm on}~\{ x\in M\backslash U \mid r(x)\ge r_0 \} \, 
  \backslash \, {\mathcal Cut}(\partial U).
\end{align}
Then, there exists a constant $C_2 >0$ such that 
$$
  u(x) \ge C_2 \,\exp \big( - \alpha_0 r(x) \big) \quad {\rm for}~~x \in M. 
$$ 
\end{thm}

We can explicitly construct a ground state with various decay orders. 
For example, we have the following: 
\begin{prop}
Let $\kappa$ be a positive constant. 
Then, for any $\lambda_0 \in (0, \frac{(n-1)^2\kappa}{4})$, there exists a rotationally symmetric Riemannian manifold $(M, g) = (\mathbb{R}^n, dr^2 + f_{\lambda_0}(r)^2 g_{S^{n-1}(1)})$ such that the following {\rm (i) and (ii)} hold\,{\rm :}
\begin{enumerate}[{\rm $($i$)$}]
\item $f_{\lambda_0}(r) = C_{r_0} \cdot \exp \{ \sqrt{\kappa}r \}$ for $r\ge r_0$, and hence, 
$$
  \Delta_g r \equiv (n-1) \kappa \quad {\rm for}~r \ge r_0, \quad 
  {\rm and} \quad 
  \sigma_{\rm ess}(-\Delta_g) = \big[ \frac{(n-1)^2\kappa}{4}, \infty \big);
$$
\item $\lambda_0$ is an eigenvalue of $-\Delta_g$ and $\lambda_0 = \min \sigma (-\Delta_g)$. Moreover, a ground state $u_0$, that is, a positive-valued eigenfunction with eigenvalue $\lambda_0$ satisfies
\begin{align*}
  u_0(x) = \exp \big( - a_0 \hspace{0.1mm} r(x) \big) 
  \quad {\rm for}~x~{\rm with}~r(x) \ge r_0.
\end{align*}
\end{enumerate}
Here, $r_0 = r_0(\lambda_0,n)$ is a positive constant depending only on $\lambda_0$ and $n${\rm ;} $C_{r_0}$ is a constant depending only on $r_0${\rm ;} $a_0 := \frac{(n-1)\sqrt{\kappa}}{2} + \sqrt{\frac{(n-1)^2\kappa}{4}-\lambda_0}$ is a constant. 
\end{prop}

In case the growth rate of the measure is negative, the important point to note is that we do not require the assumption that the outward normal exponential map induces a diffeomorphism. 
Indeed, we have the following: 
\begin{thm}
Let $(M, g)$ be an $n$-dimensional noncompact complete Riemannian manifold, and $U$ be an {\rm (}not necessarily relatively compact\,{\rm )} open subset with compact $C^{\infty}$-boundary $\partial U$. 
Let $\beta > 0$, $c_1 > 0$, $\delta > 0$, and $0 \le \lambda < \frac{\beta^2}{4}$ be constants, and assume that 
\begin{align*}
  \Delta_g r \le - \beta + \frac{c_1}{r\big( \log (r+1) \big)^{1 + \delta} +1} 
  \qquad 
  {\rm for}\ r \gg 1 \ 
  {\rm on} \ M \backslash (U \cup {\mathcal Cut} (\partial U)). 
\end{align*}
Let $u \in L^2(M\backslash U)$ be a solution of 
$$
  - \Delta_g u - \lambda u = 0 \quad {\rm on}\ M \backslash U. 
$$
Then, we have 
$$
  |u(x)| 
\le 
  c_2 \exp \big( \widetilde{\alpha}\, r(x) \big) \qquad {\rm for}\ x \in M \backslash U, 
$$ where $\widetilde{\alpha} := \frac{\beta}{2}- \sqrt{\frac{\beta^2}{4} - \lambda}$ and $c_2$ are positive constants. 
\end{thm}

For the case of negative growth rates of the measure, other versions which is similar to Proposition 2.1 below can be obtained in the same manner; but we shall omit precise statements here. 

It seems to be interesting to provide an explicit ground state below on an end $[R, \infty) \times N$ with a {\it negative} growth rate $-c$ of the measure: 
\begin{prop}
Let $\lambda$ and $c$ be any positive constants satisfying $0<\lambda < \frac{c^2}{4}$, and $(N, g_N)$ be an any compact connected Riemannian manifold. 
Then, there exists a warped product manifold $(\mathbb{R} \times N, g:= dt^2 + f^2(t)g_N)$ satisfying the following {\rm (i)} and {\rm (ii)} hold\,{\rm :}
\begin{enumerate}[{\rm (i)}]
\item The distance function $r := {\rm dist}_g ( \{0\} \times N, *)$ to the hypersurface $\{0\} \times N$ satisfies 
$$
  \Delta_g r =
\begin{cases}
  - c \quad & {\rm on}\ [R, \infty) \\
\ \ c \quad & {\rm on}\ (-\infty, 0],
\end{cases}
$$
where $R=R(\lambda)>0$ is a constant. 
In particular, $\sigma_{\rm ess}(-\Delta_g) = [\frac{c^2}{4}, \infty)$. 
\item There exists a ground state $\varphi$ of $-\Delta_g$ on $L^2(\mathbb{R} \times N, v_g)$ with eigenvalue $\lambda$ satisfying
$$
  \varphi (x) = 
\begin{cases}
  \exp \Big( \Big( \frac{c}{2} - \sqrt{\frac{c^2}{4} - \lambda} \,\Big) r(x) \Big) \quad & {\rm for}\ x \in [R, \infty) \times N\\
 c_1 \exp \Big( - \Big( \frac{c}{2} + \sqrt{\frac{c^2}{4} - \lambda} \, \Big) r(x)\Big)  \quad & {\rm for}\ x \in (-\infty, 0] \times N,
\end{cases}
$$
where $c_1$ is a positive constant. 
\end{enumerate}
\end{prop}
In order to apply theorems stated above, the discrete spectrum must exists below the essential spectrum; Theorem 1.4 below ensures its existence under compatible measure growth rates with those stated in Theorems above: 
\begin{thm}
Let $(M, g)$ be a complete Riemannian manifold and $W$ be a relatively compact open subset of $M$ with $C^{\infty}$-boundary $\partial W$. 
Set $r:={\rm dist}_g(W, *)$ on $M\backslash W$. 
Assume that, there exist constants $0 \neq c \in {\mathbb R}$, $\delta >0$, and $r_1>0$ such that
\begin{align*}
  \Delta_g r 
\le c - \frac{1+\delta}{2cr^2} 
  \quad {\rm on}\ \big\{ x \in M\backslash \big(W\cup {\mathcal Cut}(\partial W) \big) \mid r(x) \ge r_1 \big\}.
\end{align*}
Here, ${\mathcal Cut}(\partial W)$ stands for the cut locus of the Fermi coordinates $\exp^{\perp}_{\partial W}: {\mathcal N} ^+(\partial W) \to M\backslash W$ based on $\partial W$. 
Assume that $\min (\sigma_{\rm ess}(-\Delta_g)) = \frac{c^2}{4}$. 
Then, the discrete spectrum of $-\Delta_g$ below the constant $\frac{c^2}{4}$ is infinite\hspace{0.1mm}{\rm :} $\sharp \Big( \sigma_{\rm disc} (-\Delta_g)) \cap (0, \frac{c^2}{4}) \Big) = \infty$, where the symbol $\,\sharp$ represents the cardinality. 
\end{thm}
Theorem 1.4 slightly generalizes [10, Theorem 1.2] in the following sense: first, we assume the measure growth rate condition instead of the Ricci curvature condition; Note that, if $c<0$, the growth rate $\Delta_g r$ of the measure is negative in Theorem 1.4 and Proposition 6.1 below.

\section{Proof of Proposition 1.1}
In this section, we shall prove Proposition 2.1 below, which includes Proposition 1.1. 

\begin{prop}
\begin{sloppypar}
Let $(M,g)$ be an $n$-dimensional noncompact Riemannian manifold and $U$ be an {\rm (}not necessarily relatively compact\,{\rm )} open subset of $M$. 
Assume that the boundary $\partial U$ is compact, connected, and $C^{\infty}$ and that the outward normal exponential map $\exp_{\partial U}^{\perp} : N^{+}(\partial U) \to E : =M \backslash U$ induces a diffeomorphism, where $N^{+}(\partial U) = \{ v\in TM|_{\partial U} \mid v {\rm ~is~outward~normal~to~}\partial U \}$. 
Let $\nu$ be the outward unit normal vector field along $\partial E$, and denote $r(x) := {\rm dist}_g(\partial U, x)$ for $x \in E$. 
Then, $[0, \infty) \times \partial E \ni (r, w) \mapsto \exp _w(r\nu_w) \in E$ is a coordinates on $E$. 

Let $c>0$ be a constant, and assume that $u \in L^2 (E, v_g)$ is a solution of the eigenvalue equation 
\begin{align}
( -\Delta_g - \lambda )u = 0 \quad {\rm on}~~E, \tag{2}
\end{align}
where $0 \le \lambda < \frac{c^2}{4}$ is a constant. 
For simplicity, we shall set $\alpha := \frac{c}{2}+\sqrt{\frac{c^2}{4}-\lambda}$. 
\begin{enumerate}[{\rm (i)}]
\item If $\displaystyle \liminf_{r \to \infty} \Delta_g r \ge c$ and if $\beta_1$ is a constant satisfying 
$$
  \frac{c}{2}-\sqrt{\frac{c^2}{4}-\lambda} < \beta_1 < \frac{c}{2}+\sqrt{\frac{c^2}{4}-\lambda}, 
$$
then $\exp (- \beta_1 r)$ is a supersolution of $(2)$ near the infinity of $E$. 
In particular, there exists a constant $C_1>0$, depending on $\beta_1$, such that 
$$
  |u(x)| \le C_1 \exp \big(-\beta_1 r(x) \big) \quad {\rm for}~~x \in E.
$$
\item If $\displaystyle \lim_{r \to \infty} \Delta_g r = c$ uniformly with respect to $w \in \partial E$ and if $\beta_2$ is a constant satisfying 
$$
  0 \le \beta_2 < \frac{c}{2} - \sqrt{\frac{c^2}{4}-\lambda} \quad {\rm or} \quad \beta_2 > \frac{c}{2} + \sqrt{\frac{c^2}{4} - \lambda}, 
$$
then $\exp (- \beta_2 r)$ is a subsolution of $(2)$ near the infinity of $E$. 
\item If there exist constants $\beta_3 > 0$ and $\varepsilon > 0$ satisfying $\displaystyle c - \frac{\beta_3}{r(\log r)^{1 + \varepsilon}} \le \Delta_g r$ for $r \gg 1$, then there exists a constant $C_2 > 0$ such that 
$$
  |u(x)| \le C_2 \exp \big( -\alpha r(x) \big) \quad {\rm for}~~x \in E.
$$
\item If there exist a constant $\beta_4 > 0$ satisfying $\displaystyle c - \frac{\beta_4}{r\log r} \le \Delta_g r$ for $r \gg 1$, then there exists a constant $C_3 > 0$ such that 
$$
  |u(x)| \le C_3 (\log (r(x)+2))^{\frac{\alpha \beta_4}{2\alpha - c}} \exp \big( -\alpha r(x) \big) \quad {\rm for}~~x \in E.
$$
\item If there exist constants $\beta_5 >0$ and $0<\theta<1$ satisfying $\displaystyle c - \frac{\beta_5}{r(\log r)^{\theta}} \le \Delta_g r$ for $r \gg 1$, then there exists a constant $C_4 > 0$ such that 
$$
  |u(x)| \le C_4 \exp \Big( -\alpha r(x) + \frac{\alpha \beta_5}{(2\alpha-c)(1-\theta)} (\log (r(x)+1))^{1-\theta} \Big) \quad {\rm for}~~x \in E.
$$
\item If there exists a constant $0 < \theta \le 1$ satisfying $\displaystyle c - \frac{(2\alpha -c)\theta}{\alpha r} \le \Delta_g r$ for $r \gg 1$, then there exists a constant $C_5 > 0$ such that 
$$
  |u(x)| \le C_5 (r(x)+1)^{\theta } \exp \big( -\alpha r(x) \big) \quad {\rm for}~~x \in E.
$$
\item If there exist constants $\beta_6 > 0$, $0 < \theta < 1$, and $\varepsilon > 0$ satisfying $\displaystyle c - \frac{(2\alpha -c)(1-\theta)\beta_6 - \varepsilon}{\alpha r^{\theta}} \le \Delta_g r$ for $r \gg 1$, then there exists a constant $C_6 > 0$ such that 
$$
  |u(x)| \le C_6 \exp \big( -\alpha r(x) + \beta_6 r(x)^{1-\theta} \big) \quad {\rm for}~~x \in E.
$$
\item If there exist constants $\theta > 0$ and $\varepsilon > 0$ satisfying $\displaystyle c + \frac{(2\alpha -c)\theta + \varepsilon}{\alpha r\log r} \le \Delta_g r$ for $r \gg 1$, then there exists a constant $C_7 > 0$ such that 
$$
  |u(x)| \le C_7 \big( \log (r(x) + 2) \big)^{-\theta} \exp \big( -\alpha r(x) \big) \quad {\rm for}~~x \in E.
$$
\item If there exist constants $\beta_7 > 0$, $0 < \theta < 1$, and $\varepsilon>0$ satisfying $\displaystyle c + \frac{(2\alpha -c)(1 - \theta)\beta_7 + \varepsilon}{\alpha r (\log r)^{\theta}} \le \Delta_g r$, then there exists a constant $C_8 > 0$ such that 
$$
  |u(x)| \le C_8 \exp \big( - \alpha r(x) - \beta_7 (\log (r(x) +1))^{1-\theta} \big) \quad {\rm for}~~x \in E.
$$
\item If there exist constants $\theta > 0$ and $\varepsilon > 0$ satisfying $\displaystyle c + \frac{(2\alpha -c)\theta + \varepsilon}{\alpha r} \le \Delta_g r$, then there exists a constant $C_9 > 0$ such that 
$$
  |u(x)| \le C_9 r^{-\theta} \exp \big( - \alpha r(x) \big) \quad {\rm for}~~x \in E.
$$
\item If there exist constants $\beta_8 > 0$, $0 < \theta < 1$, and $\varepsilon>0$ satisfying $\displaystyle c + \frac{(2\alpha -c)(1-\theta)\beta_8 + \varepsilon}{\alpha r^{\theta}} \le \Delta_g r$, then there exists a constant $C_{10} > 0$ such that 
$$
  |u(x)| \le C_{10} \exp \big( - \alpha r(x) - \beta_8 r(x)^{1-\theta} \big) \quad {\rm for}~~x \in E.
$$
\end{enumerate}
\end{sloppypar}
\end{prop}
\vspace{2mm}

In order to prove Proposition 2.1, we shall introduce some terminology. 
Let $(M, g)$ be an $n$-dimensional noncompact complete Riemannian manifold and $\Omega$ be an open subset of $M$. 
We shall consider a solution $u$ of the eigenvalue equation
\begin{align}
  ( - \Delta_g - \lambda ) u = 0 
\end{align}
on $\Omega$, where $\lambda \ge 0$ is a constant. 
By a supersolution of $(4)$, we mean a real-valued function $u$ in $H_{\rm loc}^1(\Omega)$ satisfying 
\begin{align*}
  \int_{\Omega} \left\{ \langle \nabla u, \nabla \varphi \rangle - \lambda u \varphi \right\} \, dv_g \ge 0
\end{align*}
for every nonnegative-valued function $\varphi \in C^{\infty}_0 (\Omega)$; analogously, by a subsolution of $(4)$, we mean a real-valued function $u$ in $H_{\rm loc}^1(\Omega)$ such that
\begin{align*}
  \int_{\Omega} \left\{ \langle \nabla u, \nabla \varphi \rangle - \lambda u \varphi \right\} \, dv_g \le 0
\end{align*}
for every nonnegative-valued function $\varphi \in C^{\infty}_0 (\Omega)$. 
We shall recall the following:
\begin{lem}[S.~Agmon \cite{A}]
Assume that $\Omega$ is unbounded and that the boundary $\partial \Omega$ of $\Omega$ is connected and compact. 
We shall denote $r={\rm dist}_g(\partial \Omega,*)$ on $\Omega$. 
Let $w$ be a positive-valued continuous supersolution of $(4)$ on $\Omega$ and $v$ be a continuous subsolution of $(4)$ on $\Omega$. 
Assume that there exists a constant $\alpha >1$ such that
\begin{align*}
\liminf_{R\to \infty} \frac{1}{R^2} \int_{E(R, \alpha R)} v^2 \,dv_g = 0, \end{align*}
where we set $E(s, t):=\{ x \in \Omega \mid s < r(x) < t \}$ for $0<s<t$. 
Then, there exists a positive constant $C$ such that 
\begin{align*}
  v(x) \le C w(x) \quad 
  {\rm on}~\{ x \in \Omega \mid r(x) \ge 1 \}. 
\end{align*}
\end{lem}

For the proof of Lemma 2.1, see \cite{A} (or Lemma 5.1 below). 
\vspace{2mm}

The proof of Proposition 2.1 is simple. 
We shall prove only Proposition 2.1 (iii). 
Other assertions will be obtained in the same manner. 
Let $E$ be an end as is stated in Proposition 2.1 and assume that $\Delta_g r \ge c - \frac{\beta_3}{r (\log r)^{1+\varepsilon}}$ for $r \gg 1$, where $\beta_3$ and $\varepsilon$ are positive constants. 
We shall set $\delta_1 := \frac{\alpha \beta_3}{\varepsilon (2\alpha -c)}$. 
Then, a simple computation shows that 
\begin{align*}
& \exp \left\{ \alpha r + \delta_1 (\log r)^{-\varepsilon} \right\} 
  (- \Delta_g - \lambda ) \exp \{ - \alpha r - \delta_1(\log r)^{-\varepsilon}\}  \\
=&
  - \alpha ^2 
  + \big( \alpha - \varepsilon \delta_1 r^{-1}(\log r)^{-1-\varepsilon} \big) 
  \Delta_g r 
  - \lambda + 2\varepsilon \alpha \delta_1 r^{-1}(\log r)^{-1-\varepsilon} \\
& + \varepsilon \delta_1 r^{-2}(\log r)^{-1-\varepsilon} 
  + O(r^{-2}(\log r)^{-2-\varepsilon}) \\
\ge 
& - \alpha ^2 + c \alpha - \lambda 
  + \big( (2\alpha -c)\varepsilon \delta_1 - \alpha \beta_3 \big) 
  r^{-1}(\log r)^{-1-\varepsilon} \\ 
& + \varepsilon \delta_1 r^{-2}(\log r)^{-1-\varepsilon} 
  + O(r^{-2}(\log r)^{-2-\varepsilon}) \\
=
& \varepsilon \delta_1 r^{-2}(\log r)^{-1-\varepsilon} 
  + O(r^{-2}(\log r)^{-2-\varepsilon}) ,
\end{align*}
where we have used the equation $\alpha ^2 - c \alpha + \lambda = 0$. 
Thus, $\exp \{ - \alpha r - \delta_1 (\log r)^{-\varepsilon}\}$ is a supersolution of $(4)$ on some neighborhood of the infinity of $E$. 
Since $|u|$ is a subsolution of (4), Lemma 2.1 implies that 
\begin{align*}
  |u(x)| \le & C \exp \{ - \alpha r(x) - \delta_1 (\log r(x))^{-\varepsilon} \} \\ 
  \le & C \exp \{ - \alpha r(x) \}
\end{align*}
on $\{ x \in E \mid r(x) \ge 2 \}$, where $C>0$ is a constant. 
Thus, we have proved Proposition 2.1 (iii).

\section{Explicit construction of a ground state}

In this section, we shall prove Proposition $1.2$ and $1.3$. 
\vspace{2mm}


\noindent{\it Proof of Proposition 1.2}. Let $\lambda \in \big(0, \frac{(n-1)^2\kappa}{4} \big)$ be a constant. 
Then, there exists a unique constant $R_0 = R_0(\lambda, n)>0$ such that the first eigenvalue of the Dirichlet Laplacian $-\Delta^{\rm D}$ on the open ball $B_{\mathbb{R}^n}({\bf 0},R_0)$ with its radius $R_0$ in Euclidean $n$-space $(\mathbb{R}^n, g_0)$ coincides with $\lambda$. 
Let $\varphi_1$ be a positive-valued Dirichlet first eigenfunction of $B_{\mathbb{R}^n}({\bf 0},R_0)$. 
Then, since $(\mathbb{R}^n, g_0)$ is rotational symmetric, $\varphi_1$ is a radial function, that is, 
\begin{align*}
  \varphi_1 (x) = v_1 (|x|) \quad {\rm for}\ |x| \le R_0,
\end{align*}
where $|x|:={\rm dist}_{g_0}({\bf 0}, x)$; moreover, we have 
\begin{align*}
  v_1 > 0 \quad {\rm on}\ [0, R_0); \quad v_1'< 0 \quad {\rm on}\ (0, R_0].
\end{align*}
Since $v_1' < 0$ on $(0, \frac{R_0}{2}]$ and $\frac{d}{dt}\exp\{ - a_0 t \}<0$ on $[R_0,\infty)$, we can and do take a positive-valued function $v_0 \in C^{\infty}[0,\infty)$ so that 
\begin{align*}
& v_0(t) = 
\begin{cases}
  v_1(t)  & {\rm for}\ t \in [0, \frac{R_0}{2}] \\
  \exp\{ -a_0 t \} & {\rm for}\ t \in [R_0, \infty); 
\end{cases}\\
& v_0' < 0 \quad {\rm for} \ t \in (0,\infty). 
\end{align*}
Using this function $v_0$, we shall define 
\begin{align*}
  f_{\lambda}(r) := 
  \frac{R_0}{2} \exp \left\{ - \frac{1}{n-1} \int_{\frac{R_0}{2}}^r \frac{v_0''(t) + \lambda v_0(t)}{v_0'(t)} \, dt \right\} \quad {\rm for}\ r > 0.
\end{align*}
Then, since $v_1''(t) + \frac{n-1}{r} v_1'(t) = - \lambda v_1(t)$ for $t \in (0,\frac{R_0}{2})$, a direct computation shows that 
\begin{align*}
  f_{\lambda} (r) = r \quad {\rm for}~~ 0 < r \le \frac{R_0}{2} . 
\end{align*}
In particular, $(\mathbb{R}^n, g_{f_{\lambda}} := dr^2 + f_{\lambda}(r)^2 g_{S^{n-1}(1)})$ is smooth around the origin ${\bf 0}\in \mathbb{R}^n$. 
Moreover, from the definitions of $v_0$ and $f_{\lambda}$, we see that 
\begin{align}
  f_{\lambda} (r) = \frac{R_0}{2} \cdot C(R_0) \cdot \exp \{ \sqrt{\kappa}(r-R_0) \}\quad {\rm for}~~r \ge R_0,
\end{align}
where $C(R_0)$ is the constant defined by
\begin{align*}
  C(R_0) := \exp \left\{ - \frac{1}{n-1} 
  \int_{\frac{R_0}{2}}^{R_0} \frac{v_0''(t) + \lambda v_0(t)}{v_0'(t)} \,dt
  \right\}. 
\end{align*}
In particular, $\Delta_{g_{f_{\lambda}}} r \equiv (n-1) \sqrt{\kappa}$ for $r \ge R_0$, and hence, $\sigma_{\rm ess}(-\Delta_{g_{f_{\lambda}}}) = [\frac{(n-1)^2\kappa}{4}, \infty)$. 
Moreover, from the definition of $f_{\lambda}$, we have
\begin{align*}
 - v_0''(r) - (n-1) \frac{f_{\lambda}'(r)}{f_{\lambda}(r)} v_0'(r) 
 = \lambda v_0 (r) \quad {\rm on}~~\mathbb{R}^n \backslash \{{\mathbf 0}\}. 
\end{align*}
Here, since $v_0 = v_1$ on $[0, \frac{R_0}{2}]$, we have $- \Delta_{g_{f_{\lambda}}} (v_0(r)) = - \Delta (v_1 (r)) = \lambda v_1(r)$ on $B_{\mathbb{R}^n}({\bf 0}, \frac{R_0}{2})$; thus, including the neighborhood of the origin ${\bf 0}$, the following equation holds:
\begin{align*}
  - \Delta_{g_{f_{\lambda}}} (v_0(r)) = \lambda \,v_0 (r)
  \quad {\rm on}~~(\mathbb{R}^n, g_{f_{\lambda}}).
\end{align*}
From (5) together with $a_0 = \frac{(n-1)\sqrt{\kappa}}{2} + \sqrt{\frac{(n-1)^2\kappa}{4}-\lambda_0}$, we see that $v_0 \in L^2(\mathbb{R}^n, v_{g_{\lambda}})$; thus, $\lambda \in \sigma(-\Delta_{g_{f_{\lambda}}})$. 
Furthermore, since $v_0>0$ on $[0,\infty)$, the mini-max principle implies that $\lambda=\min \sigma (-\Delta_{g_{f_{\lambda}}})$ and that $v_0 (r)$ is a ground state of $(\mathbb{R}^n, g_{f_{\lambda}})$. 
Thus, we have proved Proposition 1.2. 
\vspace{3mm}

\noindent{\it Proof of Proposition 1.3}. Let $\lambda$ and $c$ be any positive constants satisfying $0<\lambda < \frac{c^2}{4}$. 
We shall set $u_1(t) := \sin \big( \sqrt{\lambda} \,t \big)$ and $R_1 := \frac{\pi}{2 \sqrt{\lambda}}$. 
We shall take a small positive constant $\varepsilon$ so that $\varepsilon < R_1 - \varepsilon$; note that $u_1'>0$ on $[\varepsilon, R_1 - \varepsilon]$. 
First, we shall take positive constants $c_1$ so that
$$
  0 < c_1  < u_1(\varepsilon). 
$$
Then, we shall take a positive-valued $C^{\infty}$-function $\varphi$ on $\mathbb{R}$ satisfying 
$$
  \varphi' > 0 \quad {\rm on}\ \mathbb{R}
$$
and 
\begin{align*}
\varphi (t) = 
\begin{cases}
  c_1 \exp \Big\{ \Big(\frac{c}{2} + \sqrt{\frac{c^2}{4} - \lambda} \Big) \, t \Big\} \quad & {\rm for} \ t\in (-\infty, 0] \\
  u_1(t) \quad & {\rm for} \ t\in [\varepsilon, R_1 - \varepsilon] \\
  \exp \Big\{ \Big(\frac{c}{2} - \sqrt{\frac{c^2}{4} - \lambda } \Big) \, (t- R_1) \Big\} \quad & {\rm for} \ t\in [R_1, \infty) .
\end{cases}
\end{align*}
Here, note that $u_1(R_1-\varepsilon) < 1= \exp \Big\{ \Big(\frac{c}{2} - \sqrt{\frac{c^2}{4} - \lambda } \Big) \, (t- R_1) \Big\}\Big|_{t=R_1}$. 
Using this function $\varphi $, we shall set 
\begin{align*}
  f(t) := \exp \left( - \int_0^t \frac{\varphi''(s) + \lambda \varphi(s)}{(n-1) \varphi'(s)} \right) ds
\end{align*}
and define a Riemannian manifold $(\mathbb{R} \times N, g := dt^2 + f^2(t) g_N)$, where $t$ is the standard coordinates on $\mathbb{R}$; $(N, g_N)$ is any given compact Riemannian manifold. 
Let $r$ denote the distance function to the hypersurface $\{0\}\times N$, i.e., $r(x) ={\rm dist}_g(\{0\}\times N, x)$ for $x \in \mathbb{R} \times N$. 
Then, direct computations shows that this $(\mathbb{R} \times N, g
=dt^2 + f^2(t) g_N)$ satisfies the desired assertions (i) and (ii) with $R:=R_1-\varepsilon$. 
Thus, we have proved Proposition 1.3.

\section{Proof of Theorem 1.2 and 1.3}

Let $(M, g)$ be an $n$-dimensional noncompact connected complete Riemannian manifold, and $U$ be an (not necessarily relatively compact) open subset of $M$ with compact $C^{\infty}$ boundary $W := \partial U$. 
We shall set $E := M \backslash U$ and consider the outward normal exponential map $\exp^{\perp}_W: {\mathcal N}^+(W) \to E$, where ${\mathcal N}^+(W)$ stands for the outward normal bundle on $W$. 
Let $\nu$ denote the outward unit normal vector field along $W$. 
We shall set 
\begin{align*}
  \rho (\theta ) 
& := 
  \sup \left\{ 
  t > 0 \mid {\rm dist}_g 
  \left( U, \exp_{W}^{\perp} \left( t \cdot \nu (\theta ) \right) \right) = t 
  \right\} \quad {\rm for}\ \theta \in W;\\
  {\mathcal D}_W
& := \left\{ t\cdot \nu (\theta) \in \mathcal{N}^+(W) \mid 
  0 \le t < \rho (\theta), \, \theta \in W \right\}.
\end{align*}
Then, $\exp_W^{\perp}|_{{\mathcal D}_W} : {\mathcal D}_W \to \exp_W^{\perp} \left( {\mathcal D}_W \right)= M \backslash \big( U \cup {\mathcal Cut}(W) \big)$ is a diffeomorphism. 
Let $\sigma$ denote the induced Riemannian measure on $W$, and write the Riemannian measure $dv_g$ on $\exp_W^{\perp} \left( {\mathcal D}_W \right)= M \backslash \big( U \cup {\mathcal Cut}(W) \big)$ as follows:
\begin{align*}
  dv_g = \sqrt{g^{W}(r, \theta )}\,dr \, d\sigma 
  \qquad {\rm for}~~\theta \in W \ {\rm and} \ 0 \le r < \rho(\theta ). 
\end{align*}
The following fundamental lemma is easily obtained by integration-by-parts, and hence, we shall omit its proof: 
\begin{lem}
On the understanding stated above, if $w(r)$ is a function depends only on $r$, and if $\varphi \in C_0^{\infty}({\rm Int}(E))$ be a nonnegative valued function, then we have
\begin{align*}
 & \int_E \langle \nabla w , \nabla \varphi \rangle \,dv_g - \lambda \int_E w \varphi \,dv_g \\
=
 & \int_W (\partial_r w)(\rho(\theta))\, \varphi (\rho(\theta), \theta) \sqrt{g^W(\rho(\theta), \theta)} \,d\sigma (\theta) + \int_E \left( -\Delta_g w - \lambda w \right) \varphi \,dv_g .
\end{align*}
Here, ${\rm Int}(E):= E\backslash W$ stands for the interior of $E$. 
\end{lem}
\noindent{\it Proof of Theorem 1.3.} 
Assume that
\begin{align}
  \Delta_g r \le - \beta + \frac{c_2}{r ( \log r )^{1 + \delta} } 
  \qquad 
  {\rm for}\ r \ge r_0 \ 
  {\rm on} \ M \backslash (U \cup {\mathcal Cut} (W)), 
\end{align}
where $r_0 >2$ and $c_2>$ are constants. 
Let $\varepsilon$ is a constant satisfying $0<\varepsilon<\delta$. 
Then, direct computations show that
\begin{align}
& \exp \big( - \widetilde{\alpha} r + (\log r)^{-\varepsilon} \big) (- \Delta_g - \lambda )   \exp \big( \widetilde{\alpha} r - (\log r)^{-\varepsilon} \big) \\ 
=
& - \widetilde{\alpha}^2 - 2 \varepsilon \widetilde{\alpha} r^{-1}(\log r)^{ -1 -\varepsilon} 
  - \big( \widetilde{\alpha} + \varepsilon r^{-1} (\log r)^{ - 1 - \varepsilon} \big) \Delta_g r  \notag \\  
& - \lambda + O\big( r^{-2}(\log r)^{ -1 -\varepsilon} \big) \notag \\
\ge 
& - \widetilde{\alpha}^2 + \beta \widetilde{\alpha} - \lambda 
  + (\beta - 2\widetilde{\alpha}) \varepsilon r^{-1}(\log r)^{-1 -\varepsilon} 
  - c_2 \widetilde{\alpha} r^{-1}(\log r)^{-1-\delta} \notag \\
& + O\big( r^{-2}(\log r)^{-1-\varepsilon} \big) \notag \\
=
&  \sqrt{\beta^2 -4 \lambda}\,r^{-1}(\log r)^{-1-\varepsilon} 
  - c_2 \widetilde{\alpha} r^{-1}(\log r)^{-1-\delta} + O(r^{-2}) 
  > 0 \notag \\ 
& \hspace{20mm} {\rm for} \ r \ge r_1 \ {\rm on}\ 
  M \backslash (U \cup {\mathcal Cut} (\partial U)), \notag 
\end{align}
where we have used the equation $- \widetilde{\alpha}^2 + \beta \widetilde{\alpha} - \lambda = 0$ and inequality (6); $r_1 \gg r_0$ is a constant depending only on the constants stated above. 
We shall apply Lemma 4.1 with $w = \exp \big( \widetilde{\alpha} r - (\log r)^{-\varepsilon} \big)$. 
Then, since $\partial_r w \ge 0$ for $r \gg 1$, Lemma 4.1, combined with (7), implies that $\exp \big( \alpha r - (\log r)^{-\varepsilon} \big)$ is a supersolution of (4) in the sense of distribution on some neighborhood of the infinity of $M\backslash U$. 
Hence, by Lemma 2.1, we obtained the desired result in Theorem 1.3.
\vspace{3mm}

\noindent{\it Proof of Theorem 1.2.} We shall take a constant $\varepsilon$ so that $0<\varepsilon < \delta$, and consider a function $\exp \left( - \alpha_0 r + (\log r)^{-\varepsilon} \right)$ on $M\backslash U$. 
Then, $\exp \left( - \alpha_0 r + (\log r)^{-\varepsilon} \right) \in H^1_{{\rm loc}} \big( B(2,\infty) \big)$; here, for $t\ge 0$, we set $B(t, \infty):=\{ x \in M\backslash U \mid r(t) \ge t \}$ for simplicity. 
Then, simple calculations show that
\begin{align}
& \exp \left( \alpha_0 r - (\log r)^{-\varepsilon} \right) (- \Delta_g -\lambda) 
  \exp \left( - \alpha_0 r + (\log r)^{-\varepsilon} \right) \\
= 
& - \alpha_0^2 - \lambda - 2 \alpha_0 \varepsilon r^{-1} (\log r)^{-1-\varepsilon} 
  + \left( \alpha_0 + \varepsilon r^{-1} (\log r)^{-1-\varepsilon} \right) \Delta_g r 
  \notag \\
& - O \left( r^{-2} (\log r)^{-\varepsilon-1} \right) \notag \\
\le 
& - \alpha_0^2 - \lambda - 2 \alpha_0 \varepsilon r^{-1} (\log r)^{-1-\varepsilon} 
  \notag \\
& + \left( \alpha_0 + \varepsilon r^{-1} (\log r)^{-1-\varepsilon} \right) 
  \left( c + C_1 r^{-1}(\log r)^{-1-\delta} \right) 
  - O \left( r^{-2} (\log r)^{-\varepsilon-1} \right) \notag \\ 
=
& - \alpha_0^2 + c \alpha_0 - \lambda 
  + (c - 2\alpha_0) \varepsilon r^{-1} (\log r)^{-1-\varepsilon} 
  + \alpha_0 C_1 r^{-1}(\log r)^{-1-\delta} \notag \\ 
& - O \left( r^{-2} (\log r)^{-\varepsilon-1} \right) \notag \\ 
=
& - \sqrt{c^2 -4 \lambda} \varepsilon r^{-1} (\log r)^{-1-\varepsilon}
  + \alpha_0 C_1 r^{-1}(\log r)^{-1-\delta} 
  - O \left( r^{-2} (\log r)^{-\varepsilon-1} \right) \notag \\ 
& \hspace{50mm} 
  \quad {\rm on}~~B(2,\infty) \backslash {\mathcal Cut}(W) , \notag
\end{align}
where we have used the assumption (3), and equations, $- \alpha_0^2 + c \alpha_0 - \lambda = 0$ and $c - 2\alpha_0 = - \sqrt{c^2 -4 \lambda}$. 
Thus, we obtain 
\begin{align}
 (- \Delta_g -\lambda) 
  \exp \left( - \alpha_0 r + (\log r)^{-\varepsilon} \right) < 0 
  \quad {\rm on}\ B(r_1, \infty) \backslash {\mathcal Cut}(W), 
\end{align}
where $r_1 \gg \min \{r_0, 2\}$ is a large constant. 
Since $\exp \left( - \alpha_0 r + (\log r)^{-\varepsilon} \right)$ is a decreasing function of $r \gg 1$, Lemma 4.1, combined with (9), implies that $\exp \left( - \alpha_0 r + (\log r)^{-\varepsilon} \right)$ is a subsolution of (4) on some neighborhood of the infinity of $M\backslash U$. 

Next, we shall show that $\exp \left( - \alpha_0 r + (\log r)^{-\varepsilon} \right) \in L^2 \big(B(2,\infty)\big)$. 
Under the notations stated at the beginning of this section, $\Delta_g r = \frac{\partial _r \sqrt{g^W(r, \theta)}}{\sqrt{g^W(r, \theta)}}$ on $M \backslash \big( U \cup {\mathcal Cut}(\partial U) \big)$. 
Hence, the assumption $(3)$, together with the fact $\alpha_0 > \frac{c}{2}$, implies that $\exp \left( - \alpha_0 r + (\log r)^{-\varepsilon} \right) \in L^2 \big(B(2,\infty)\big)$. 

Hence, Lemma 2.1 implies that, there exists a constant $C_5>0$ such that
\begin{align*}
 u & \ge C_5 \exp \left\{ - \alpha_0 r + (\log r)^{-\varepsilon} \right\} \\
   & \ge C_5 \exp \left\{ - \alpha_0 r \right\} \qquad {\rm on}~~B(2,\infty). 
\end{align*}
Thus, we have proved Theorem 1.2. 

\section{Proof of Theorem 1.1}

We shall first prove Theorem 5.1 below, from which Theorem 1.1 follows. 
\vspace{2mm}

Let $(M, g)$ be an $n$-dimensional noncompact complete Riemannian manifold. 
Let $p \in M$ be a fixed point and write $r(x):={\rm dist}_g(p, x)$ for $x \in M$. 
Let $E_0$ be one of unbounded components of $M\backslash B_p(R)$, where $R>0$ is a constant and $B_p(R):=\{ x \in M \mid r(x) < R \}$. 
Since $E_0$ is unbounded, there exists at least one normal geodesic ray $\gamma: [0, \infty) \ni t \mapsto \exp_p (tv_0) \in M$ satisfying $r(\gamma(t)) = t$ for $t \in [0,\infty)$ and $\gamma([R, \infty)) \subset E_0$. 
Here, $v_0 \in U_pM := \{ v \in T_pM \mid |v|=1 \}$. 
For $0<\theta<\pi$, let ${\mathcal R}(v_0, \theta)$ denote an open cone in $T_pM$ defined by ${\mathcal R}(v_0, \theta) := \{ v \in T_pM \mid \measuredangle (v_0, v) < \theta \}$. 
Assume that
\begin{enumerate}[(a)]
\item there exists a constant $\theta_0 >0$ such that 
$$\exp_p|_{{\mathcal R}(v_0, \theta)} : {\mathcal R}(v_0, \theta) \to \exp_p ({\mathcal R}(v_0, \theta)) =: {\mathcal A}(\theta_0)$$ is a diffeomorphism; 
\item there exists a constant $\kappa_0 >0$ such that $K_g(\partial_r \wedge v) \le - \kappa_0$ for any $x \in {\mathcal A}(\theta_0)$ and $v\in T_xM$ with $\partial_r \perp v$. Here, $K_g$ stands for the sectional curvature of $(M ,g)$; 
\item ${\rm Ric}_g \ge - (n-1) b$ on ${\mathcal A}(\theta_0)$, here $b \gg \kappa_0$ is a constant; 
\item there exists a constant $\varepsilon_0 >0$ such that ${\rm inj}(x) > \varepsilon_0$ for any $x \in {\mathcal A}(\theta_0)$, where ${\rm inj}(x)$ stands for the injectivity radius at $x$. 
\end{enumerate}


\begin{thm}
Let $(M, g)$ be an $n$-dimensional noncompact Riemannian manifold satisfying {\rm (a)}, {\rm (b)}, {\rm (c)}, and {\rm (d)}, stated above. 
Let $c>0$ and $0 \le \lambda < \frac{c^2}{4}$ are constants, and assume that 
$u$ is a solution of 
\begin{align*}
  -\Delta_g u - \lambda u = 0 
  \quad {\rm on}\ \ \overline{{\mathcal A}(\theta_0)}
\end{align*}
satisfying $\int_{{\mathcal A}(\theta_0)} u^2 dv_g < \infty$. 
For simplicity, we shall set $\alpha := \frac{c}{2} + \sqrt{\frac{c^2}{4}-\lambda}$. 
Assume further that, 
\vspace{0.5mm} 

\noindent {\rm (e)} there exist constants $\delta_1$, satisfying $\frac{c}{2} - \sqrt{\frac{c^2}{4}-\lambda} < \delta_1 < \alpha $, and $C_0 > 0$ such that 
\begin{align}
  |u(x)| \le C_0 \exp (- \delta_1 r (x)) 
  \quad {\rm for}~~x \in {\mathcal A}(\theta_0). 
\end{align}
Under these assumptions, if we replace $E$ with $\overline{{\mathcal A}(\theta_0)}$ in Proposition $1.1$, then we obtain the same upper bounds of $|u|$ on $\exp_p ({\mathcal R}(v_0, \theta_1))$ for any $\theta_1 \in (0, \theta_0)$ as those stated in {\rm (i)} and {\rm (iii)--(xi)} of Proposition $2.1$ above. 
\end{thm}

In order to prove Theorem 5.1, we shall first slightly extend Lemma $2.1$ as follows:
\begin{lem}
Let $\Omega$ and $\Omega_0$ be unbounded open subsets of a noncompact complete Riemannian manifold $(M, g)$ satisfying $\overline{\Omega_0} \subseteq \Omega$. 
Let $\lambda \ge 0$ be a constant and $u$ be a solution of $(4)$ on $\Omega$. 
Let $w \in C^0(\overline{\Omega})$ be a supersolution of $(4)$ on $\Omega$ satisfying $w > 0$ on $\Omega$. 
Assume that there exists a constants $C > 0$ such that
\begin{align*}
  C w(x) - u(x) > 0 \ \ {\rm for}\ \  
  x \in \Omega \backslash \Omega_0. 
\end{align*}
Assume further that there exists a constant $\alpha > 1$ such that 
\begin{align}
  \liminf_{n \to \infty} \frac{1}{n^2}\int_{\Omega (n, \alpha n)} u^2 \, dv_g = 0, 
\end{align}
where $\Omega (n, \alpha n) := \{ x \in \Omega \mid n \le {\rm dist}_g (p_0, x) \le \alpha n \}$ and $p_0$ is a fixed point of $M$. 
Then, we have 
\begin{align*}
  u(x) \le C w(x) \quad {\rm for}~~x \in \Omega. 
\end{align*}
\end{lem}
\begin{proof}
It is easy to see that Lemma 5.1 holds by modifying the arguments of Agmon \cite{A}. 
But, for the convenience of readers, we shall give its proof. 

We shall set $u_0(x) := \max \{ u(x) - C w(x), 0 \}$ for $x \in \Omega$. 
Then, $u_0$ is a subsolution of (4) on $\Omega$, and $u_0|_{\Omega \backslash \Omega_0} \equiv 0$. 
Therefore, by setting $u_0(x) =0$ for $x \in M \backslash \Omega$, $u_0$ can be extended to be a function defined on the whole $M$. 
Thus, for any real-valued $\zeta \in C^{\infty}_0(M)$, we have 
\begin{align*}
 \int_{\Omega} \big\{ \langle \nabla u_0, \nabla (\zeta ^2 u_0) \rangle - \lambda u_0 (\zeta^2u_0) \big\} dv_g \le 0, 
\end{align*}
and hence, 
\begin{align}
 \int_{\Omega} \big\{ |\nabla (\zeta u_0)|^2 - \lambda \zeta^2 u_0^2 \big\} dv_g \le \int_{\Omega} |\nabla \zeta|^2 u_0^2 \,dv_g . 
\end{align}
On the other hand, since $w$ is a positive supersolution of (4) on $\Omega$, for any real-valued $\zeta \in C^{\infty}_0(M)$, 
\begin{align*}
 \int_{\Omega} \Big\{ \Big\langle \nabla w, \nabla \frac{(\zeta u_0)^2}{w} \Big\rangle - \lambda w \frac{(\zeta u_0)^2}{w} \Big\} dv_g \ge 0, 
\end{align*}
and hence, 
\begin{align}
 \int_{\Omega} w^2 \Big| \nabla \Big(\frac{\zeta u_0}{w}\Big) \Big|^2 dv_g 
 \le 
 \int_{\Omega} |\nabla (\zeta u_0)|^2 dv_g 
 - \lambda \int_{\Omega} (\zeta u_0)^2 dv_g. 
\end{align}
From (12) and (13), we have for any real-valued $\zeta \in C^{\infty}_0(M)$, 
\begin{align}
 \int_{\Omega} w^2 \Big| \nabla \Big(\frac{\zeta u_0}{w}\Big) \Big|^2 dv_g 
 \le 
 \int_{\Omega} |\nabla \zeta|^2 u_0^2 \,dv_g. 
\end{align}
Now, let $\alpha >1$ be a constant, and define $\varphi : [0,\infty) \to [0,\infty)$ by
\begin{align*}
 \varphi (t) := 
 \begin{cases}
  \ \ \ \ 1 \quad & {\rm for}~~0 \le t \le 1, \vspace{1mm} \\ 
  \displaystyle - \frac{t-1}{\alpha - 1} + 1 \quad 
  & {\rm for}~~1 \le t \le \alpha, \vspace{1mm} \\
  \ \ \ \ 0 \quad & {\rm for}~~t \ge \alpha. 
 \end{cases}
\end{align*}
Setting $\zeta_n(x) := \varphi \Big( \frac{{\rm dist}_g(p_0, x)}{n} \Big)$ for $0 < n \in {\mathbb Z}$, substituting $\zeta = \zeta_n$ in (14) and letting $n \to \infty$, we obtain 
\begin{align}
& \int_{\Omega} w^2 \Big| \nabla \Big(\frac{u_0}{w}\Big) \Big|^2 dv_g 
  = \int_{\Omega} \lim_{n \to \infty} w^2 \Big| \nabla \Big(\frac{\zeta_n u_0}{w}\Big) \Big|^2 dv_g \\
\le 
& \liminf_{n \to \infty} \int_{\Omega} w^2 \Big| \nabla \Big(\frac{\zeta_n u_0}{w}\Big) \Big|^2 dv_g \le \liminf_{n \to \infty} \int_{\Omega} |\nabla \zeta_n|^2 u_0^2 \,dv_g \notag \\ 
= 
& \frac{1}{(\alpha -1)^2}\liminf_{n \to \infty} \frac{1}{n^2} \int_{\Omega (n, \alpha n)} u_0^2 \, dv_g = 0. \notag 
\end{align}
In the last equality, we have used our assumption (11). 
Since $w >0$ on $\Omega$, (15) implies that there exists a constant $C_0$ such that $u_0 = C_0 w$ on $\Omega$. 
Since $u_0=0$ on $\Omega \backslash \Omega_0$ and $w>0$ on $\Omega$, the constant $C_0$ must be zero. 
Thus, we have $u_0 \equiv 0$ on $\Omega$, which implies our conclusion, $u \le C w$ on $\Omega$. 
\end{proof}

Now, we shall prove Theorem 1.1 for the case corresponding Proposition 2.1 (i). 
Other cases can be shown in the same way. 

Let $\beta_1$ is a constant satisfying $\frac{c}{2}-\sqrt{\frac{c^2}{4}-\lambda} < \beta_1 < \alpha = \frac{c}{2}+\sqrt{\frac{c^2}{4}-\lambda}$. 
Since $\beta_1$ is less than $\alpha$, it suffices to prove that, for any $0<\theta_1 < \theta_0$, there exist a constant $C_{11}>0$ such that 
\begin{align}
 |u(x)| \le C_{11} \exp ( \beta_1 r (x) ) \quad 
 {\rm for}\ \ x \in \exp_p ({\mathcal R}(v_0, \theta_1)) 
 \ \ {\rm with}\ \ r(x) \gg 1
\end{align}
under the assumption that 
\begin{align}
  \Delta_g r (x) \ge c 
  \quad {\rm for}\ \ x \in \exp_p ({\mathcal R}(v_0, \theta_0)) \ \ {\rm with}\ \ r(x) \gg 1. 
\end{align}
Let $\theta_1 $ be any constant satisfying 
\begin{align}
  0 < \theta_1 < \theta_0 , 
\end{align}
and let $\eta : [0, \infty) \to [0, 1]$ be a non-increasing $C^{\infty}$-function satisfying 
\begin{align}
 \eta (t) := 
 \begin{cases}
  \ 1 \quad & {\rm for}~~0 \le t \le \theta_1, \vspace{1mm} \\
  \ 0 \quad & {\rm for}~~t \ge \displaystyle \frac{\theta_1 + \theta_0}{2}. 
 \end{cases}
\end{align}
Let $(r, \theta )$ be coordinates on $\exp_p({\mathcal R}(v_0, \theta_0))$ defined by $(0, \infty) \times \left( {\mathcal R}(v_0, \theta_0)\cap U_pM \right) \ni (r, \theta) \mapsto \exp_p(r \theta) \in \exp_p({\mathcal R}(v_0, \theta_0))$, where $U_pM := \{ v \in T_pM \mid |v|=1 \}$. 
Using this coordinates, we shall define a $C^{\infty}$-function $\varphi : \exp_p({\mathcal R}(v_0, \theta_0)) \to [0,1]$ by 
\begin{align}
  \varphi (r, \theta) := \eta \big( \measuredangle (v_0, \theta) \big) .
\end{align}
Now, let $\chi : [0, \infty) \to [0,1]$ be a $C^{\infty}$-function satisfying 
\begin{align*}
 \chi (t) := 
 \begin{cases}
  \ 1 \quad & {\rm for}~~0 \le t \le \displaystyle \frac{\varepsilon_0}{2}, \vspace{2mm} \\
  \ 0 \quad & {\rm for}~~t \ge \displaystyle \frac{3\varepsilon_0}{4}, 
 \end{cases}
\end{align*}
and set
\begin{align*}
 \widetilde{\varphi}(x) 
:= 
 \frac{\int_M \chi({\rm dist}_g(x, y)) \varphi(y) dv_g(y)}{\int_M \chi({\rm dist}_g(x,y)) dv_g(y)} \quad 
 {\rm for}~~x \in \exp_p({\mathcal R}(v_0, \theta_0))~~{\rm with}~~ r(x) \gg 1.
\end{align*}
Then, we have
\begin{align}
  |\Delta_g \widetilde{\varphi}| \le \frac{C_{12}}{\sinh (\sqrt{\kappa_0} r)} 
  \quad {\rm and} \quad 
  |\nabla \widetilde{\varphi}| \le \frac{C_{13}}{\sinh (\sqrt{\kappa_0} r)} ,
\end{align}
where $C_{12}$ and $C_{13}$ are constants depending only on $n= \dim M$, $b$, $\chi$, and the Lipchitz constant of $\eta$ (see Anderson-Schoen \cite{A-S}). 

Let $a_1$, which will be determined later, be a constant satisfying 
\begin{align}
  \delta_1< a_1 < \alpha ,
\end{align}
where $\delta_1$ is the constant stated in the assumption (e). 
We shall set
$$
  w := \widetilde{\varphi} \exp ( - a_1 r ) 
  + (1 - \widetilde{\varphi}) \exp ( - \delta_1 r) . 
$$
Then, direct computations, combined with (17) and (21), yield 
\begin{align*}
& - \Delta_g w - \lambda w \\
\ge 
& \Big\{ -\Delta_g \widetilde{\varphi} 
  + 2 a_1 \frac{\partial \widetilde{\varphi}}{\partial r} 
  + \widetilde{\varphi} ( - a_1^2 + c a_1 - \lambda ) \Big\} \exp ( - a_1 r )\\
& + \Big\{ \Delta_g \widetilde{\varphi} 
  - 2 \delta_1 \frac{\partial \widetilde{\varphi}}{\partial r} 
  + (-\delta_1^2 + c \delta_1 - \lambda)(1 - \widetilde{\varphi})  \Big\} 
  \exp (- \delta_1 r) \\
\ge 
& \exp (-\delta_1 r) \Big\{ - \frac{C_{12}+2\delta_1C_{13}}{\sinh(\sqrt{\kappa_0} r)}   + (-\delta_1^2 + c\delta_1 - \lambda)(1 - \widetilde{\varphi}) \\
& \hspace{25mm} + \frac{- a_1^2 + c a_1 - \lambda }{\exp \big\{ (a_1-\delta_1)r \big\}} 
  \widetilde{\varphi} - \frac{C_{12}+2a_1C_{13}}{\sinh(\sqrt{\kappa_0} r)\exp \big\{ (a_1-\delta_1)r \big\}} \Big\}
\end{align*}
for $r \gg 1$ on $\exp_p({\mathcal R}(\theta_0))$. 
Because $-\delta_1^2 + c\delta_1 - \lambda >0$ and $- a_1^2 + c a_1 - \lambda >0$, the right hand side of the inequality above is positive for $r \gg 1$ on $\exp_p({\mathcal R}(\theta_0))$, if $(0<)\,a_1-\delta_1< \sqrt{\kappa_0}$. 
On the other hand, from (19) and (20), $w$ coincides with $\exp (-\delta_1 r)$ on ${\mathcal R}(\theta_0) \backslash {\mathcal R}(\theta_1)$ for large $r$, and $\exp (-\delta_1 r)$ is a supersolution of (4) satisfying the assumption (e). 
Therefore, applying Lemma 4.1, we obtain
\begin{align}
  |u(x)| \le C_{14} \exp (- a_1 r(x)) \quad {\rm for}\ \ x \in \exp_p({\mathcal R}(v_0, \theta_1)) 
\end{align}
for some constant $C_{14}>0$. 
This argument holds good as long as $a_1 < \delta_1 + \sqrt{\kappa_0}$. 
Therefore, in view of (10) and (23), we can replace ``$\delta_1$ and $a_1$'' with ``$a_1$ and $a_2 := a_1 + \varepsilon$'' respectively, and obtain 
$$
  |u(x)| \le C_{15} \exp (- a_2 r(x)) \quad {\rm for}\ \ x \in \exp_p({\mathcal R}(v_0, \theta_2)) ,
$$
if $0< \varepsilon <\sqrt{\kappa_0}$. 
Here, $\theta_2$ is any fixed constant satisfying $0<\theta_2<\theta_1$, and $C_{15}>0$ is a constant. 
Thus, repeating this argument, we see that, for any $\theta_1 \in (0,\theta_0)$, our desired inequality (16) holds for some positive constant $C_{11}$. 
\vspace{2mm}

Next, we shall prove Theorem 1.1. 
In view of Theorem 5.1, it suffices to prove (e) under the assumptions of Theorem 1.1. 
In order to prove (e), we shall recall the following: 
\begin{lem}[\cite{D1}, Corollary 5.7]
Let $(M, g)$ be a complete Riemannian manifold with Ricci curvature bounded from below and $p_0$ be a fixed point of $M$. 
We write $r:={\rm dist}_g(p_0, *)$. 
Assume that $u$ is a $L^2$-solution of $-\Delta_g u - \lambda u =0$ on $M$, where $0<\lambda<\min \sigma_{\rm ess}(-\Delta_g)$ is a constant. 
Then, for any $0<\varepsilon<\sqrt{\min \sigma_{\rm ess}(-\Delta_g)-\lambda}$, there exists a positive constant $C_0(\varepsilon)$ such that 
\begin{align}
  |u(x)| 
\le 
  C_0(\varepsilon) \,{\rm Vol}(B_x(1))^{-\frac{1}{2}} 
  \exp \Big\{ \Big(\sqrt{\min \sigma_{\rm ess}(-\Delta_g)-\lambda}-\varepsilon\Big)r(x) \Big\}
\end{align}
\end{lem}
Lemma 5.2 follows from the upper bound for the heat kernel and the Agmon's proceedure \cite{A-book} (or Theorem ALW); for a proof, see [6, Section 5]. 

By the Bishop's comparison theorem combining (24) we have, for any $0 < \delta < \sqrt{\min \sigma_{\rm ess}(-\Delta_g) - \lambda}$ and $\theta_1 \in (0, \theta_0)$, 
$$
  |u(x)| \le c(\delta, \theta_1, n) \exp (-\delta r(x))
$$ 
holds on $\varphi \left(\exp_q^N \big( {\mathcal R}(v_1^N, \theta_1) \big)\right)$. 
In view of the assumption 3 in Theorem 1.1, we see that (e) holds. 
Thus, we have proved Theorem 1.1.

\section{Proof of Theorem 1.3}
First, we shall prove the following: 
\begin{prop}
Let $({\mathbb R}^n, g:=dr^2 + f(r)^2 g_{S^{n-1}(1)})$ be a rotationally symmetric Riemannian manifold. 
Assume that, there exist constants $\delta >0$, $r_1>0$, and $0 \neq c \in {\mathbb R}$ such that 
\begin{align}
  \Delta_g r = c - \frac{1+\delta}{2cr^2} \quad {\rm for}\ r \ge r_1.
\end{align}
Then, there exists a constant $r_0=r_0(\delta, c)$ such that,
\begin{align*}
  \lambda_1^D \Big( E(R, 2kR) \Big) < \frac{c^2}{4} 
  \quad {\rm if}\ R \ge \max\{ r_0, r_1 \} \ {\rm and}\  
  k > 2 \max \Big\{ 1, \exp \Big( \frac{24}{\delta} \Big) \Big\} .
\end{align*}
Here, $\lambda_1^D(\Omega)$ stands for the first Dirichlet eigenvalue of $-\Delta_g$ on a bounded domain $\Omega \subset {\mathbb R}^n${\rm ;} for $0<s<t\le \infty$, $E(s, t):= \{ x \in {\mathbb R}^n \mid s < {\rm dist}_g({\mathbf 0}, x) < t \}$ is the `annular' region centered at the origin ${\mathbf 0} \in {\mathbb R}^n$. 
In particular, $\lambda_0\big( E(R, \infty) \big) < \frac{c^2}{4}$ for any $R>0$, where $\lambda_0\big( E(R, \infty) \big):= \min \sigma (-\Delta_g^D|_{E(R, \infty)})$, and $\Delta_g^D|_{E(R, \infty)}$ stands for the Dirichlet Laplacian on $E(R, \infty)$. Note that $(25)$ implies $\sigma_{\rm ess}(-\Delta_g)=[\frac{c^2}{4}, \infty)$. 
\end{prop}

\noindent{\it Proof of Proposition 6.1.} We shall set $U:= B(r_1)$, $W:=\partial U$, and $\sqrt{g^W}(r) := f^{n-1}(r)$. 
Let $\chi$ be a function defined by
\begin{align*}
\chi(t) := 
\begin{cases} 
\ \ \frac{1}{R}(t - R)\quad & {\rm if}\quad t \in [R, 2R], \\ 
\ \ 1\quad & {\rm if}\quad t \in [2R, kR], \\ 
\ \ - \frac{1}{kR}(t - 2kR)\quad & {\rm if}\quad t \in [kR, 2kR], \\ 
\end{cases} 
\end{align*} 
where $k$ is a constant satisfying $k > 2 \max \left\{ 1, \exp \Big( \frac{24}{\delta} \Big) \right\}$; $R > \max \{ r_0, r_1 \}$ is a constant; $r_0=r_0(c, \delta)$ is a constant determined later. 
We shall set $h(r):= \left( \sqrt{g^W}(r) \right)^{-\frac{1}{2}} \cdot \chi(r) \cdot r^{\frac{1}{2}}$ and calculate 
$$
  \int_0^{\infty} \left\{ (\partial_r h)^2 - \frac{c^2}{4}  h^2 \right\} 
  \sqrt{g^W} dr
$$
as follows. 
Direct computations shows that
\begin{align*}
& \Big\{ (\partial_r h)^2 - \frac{c^2}{4}  h^2 \Big\}\sqrt{g^W} \\
= 
& \frac{1}{4} (\Delta_g r)^2\, r \,\chi^2(r) + r \left(\chi '(r)\right)^2 
  + \frac{1}{4} \chi^2(r) \,r^{-1} - (\Delta_g r) \,r \,\chi(r) \chi '(r) \\
& -\frac{1}{2} (\Delta_g r) \chi^2(r) 
  + \chi(r) \chi '(r) - \frac{c^2}{4} \chi^2(r)\, r . 
\end{align*}
Therefore, substitution of (26) into the equation above yields 
\begin{align}
& \Big\{ (\partial_r h)^2 - \frac{c^2}{4}  h^2 \Big\}\sqrt{g^W} \\
= 
& - \frac{1}{4r} \left\{ \delta - \frac{1+\delta}{cr} 
  - \frac{(1+\delta)^2}{4c^2r^2} \right\} \chi ^2(r)
  + r \left( \chi'(r) \right)^2 + \frac{1+\delta}{2cr} \chi(r) \chi'(r) 
  \notag \\
& - c r \chi(r) \chi'(r) - \frac{c}{2} \chi ^2(r) + \chi(r) \chi'(r) . \notag
\end{align}
Here, integration-by-parts yields 
\begin{align*}
& \int_0^{\infty} \left\{ - c r \chi(r) \chi'(r) - \frac{c}{2} \chi ^2(r) 
  + \chi(r) \chi'(r)  \right\} dr = 0~;~\\
& \int_0^{\infty} \frac{1+\delta}{2cr} \chi(r) \chi'(r) dr 
  = \int_0^{\infty} \frac{1+\delta}{4cr^2} \chi^2(r) dr, 
\end{align*}
and hence, the integration of (26) over $[0, \infty)$ turns out to be 
\begin{align*}
& \int_0^{\infty} 
  \Big\{ (\partial_r h)^2 - \frac{c^2}{4}  h^2 \Big\}\sqrt{g^W} \,dr \\
= 
& \int_0^{\infty} r (\chi '(r))^2 \,dr 
  - \int_0^{\infty} \frac{1}{4r} \left\{ \delta - \frac{2(1+\delta)}{cr} 
  - \frac{(1+\delta)^2}{4c^2r^2} \right\} \chi ^2(r) \, dr .
\end{align*}
As for the last term of the equation above, we have 
\begin{align*}
  - \frac{1}{4r} \left\{ \delta - \frac{2(1+\delta)}{cr} 
  - \frac{(1+\delta)^2}{4c^2r^2} \right\} \le - \frac{\delta}{8r} 
  \quad {\rm for\ any}\ r \ge r_0,
\end{align*}
where $r_0$ is a constant depending only on $\delta$ and $c$. 
Therefore, for any $R \ge \max\{r_0, r_1\}$ and $k > 2 \max \Big\{ 1, \exp \Big( \frac{24}{\delta} \Big) \Big\}$, we have
\begin{align*}
& \int_0^{\infty} 
  \Big\{ (\partial_r h)^2 - \frac{c^2}{4}  h^2 \Big\}\sqrt{g^W} \,dr \\
\le 
& \int_0^{\infty} r (\chi '(r))^2 \,dr 
  - \frac{\delta}{8} \int_{R}^{2kR} \frac{\chi^2(r)}{r} \,dr \\
<
& \frac{1}{R^2} \int_{R}^{2R} r \,dr 
  + \frac{1}{(kR)^2} \int_{kR}^{2kR} r \,dr 
  - \frac{\delta}{8} \int_{2R}^{kR} \frac{1}{r} \,dr \\
=
& 3 - \frac{\delta}{8} \log \Big( \frac{k}{2} \Big) < 0.
\end{align*}
Since ${\rm supp} ~h(r) = \overline{B(2kR)} \backslash B(R)$, the mini-max principle, together with this inequality, gives the desired assertion of Proposition 6.1. 
\vspace{3mm}

\begin{rem}
{\rm 
Let $(\mathbb{R}^n, g := dr^2 + f(r) g_{S^{n-1}(1)})$ be a rotationally symmetric manifold satisfying $f^{n-1}(r) = \exp \left(cr + \frac{1+\delta}{2cr} \right)$ for $r\ge r_0$. 
Then, $\sqrt{{\rm Vol}(E(R,R+1))} \sim C_0(c, \delta) \exp (\frac{c}{2}R)$ as $R \to \infty$. 
Moreover, by Proposition 6.1, $\sqrt{\lambda_0(E(R, \infty))-\lambda}< \sqrt{\frac{c^2}{4}  - \lambda} = \sqrt{\min \sigma_{\rm ess}(-\Delta_g)  - \lambda}$ for all $R>0$. 
Hence, by comparing Theorem ALW and Proposition 1.1 (iii), Theorem ALW seems to be not so sharp. 
Indeed, Proposition 1.1 (iii) provides pricise upper bound $\exp \{ (\frac{c}{2}+\sqrt{\frac{c^2}{4}-\lambda}) r \}$ of a solution of (4) on $(\mathbb{R}^n, g)$. 
}
\end{rem}
\vspace{2mm}

\noindent{\it Proof of Therem $1.4$} \ 
Proposition 6.1, together with the argument used in [10, Section 3], implies $\sharp \Big( \sigma_{\rm disc} (-\Delta_g)) \cap (0, \frac{c^2}{4}) \Big) = \infty$; for detail, see that section in [10].


\end{document}